\documentclass[a4paper]{article}

\usepackage{fullpage, amsmath,amssymb,amsthm,amsfonts,enumerate,textcomp, eurosym, titling, epsfig,epstopdf, tikz, esint, bigints,mathrsfs,todonotes,mathtools,verbatim,hyperref,csquotes}
\usepackage{enumerate}
\usepackage[english]{babel}
\usepackage{tikz-cd}
\input xy
\usepackage{dirtytalk}
\usepackage[
backend=biber,
style=numeric,
sorting=nty
]{biblatex}
\xyoption{all}

\DeclareMathOperator{\Res}{Res}

\def\M{\mathcal M}
\def\O{\mathcal O}
\def\ChD{\operatorname{Ch}}
\def\Ch{\operatorname{Ch}}
\def\ch{\operatorname{ch}}
\def\Ch_K{\operatorname{Ch_K}}
\def \qCh_K{\operatorname{qCh_K}}
\def \Td_K{\operatorname{Td_K}}
\def \Td{\operatorname{Td}}
\def\td{\operatorname{td}}
\def\ch{\operatorname{ch}}
\def\sdim{\operatorname{sdim}}
\def\str{\operatorname{str}}
\def\tr{\operatorname{tr}}
\def\id{\operatorname{id}}
\def\QQ{\mathbb Q}

\title{Multiplicative Quantum Cobordism Theory}
\author{Irit Huq-Kuruvilla}
\begin{document}
\maketitle

\theoremstyle{plain}
\newtheorem{thm}{Theorem}[section]
\newtheorem{lem}[thm]{Lemma}
\newtheorem{prop}[thm]{Proposition}
\newtheorem{cor}[thm]{Corollary}
\newtheorem{exercise}[thm]{Exercise}

\theoremstyle{definition}
\newtheorem{mydef}[thm]{Definition}
\newtheorem{conj}[thm]{Conjecture}
\newtheorem{exmp}[thm]{Example}
\newtheorem*{claim}{Claim}
\newtheorem*{notation}{Notation}
\newtheorem{prob}{Problem}
\newtheorem{ex}{Exercise}

\theoremstyle{remark}
\newtheorem*{rem}{Remark}
\newtheorem*{note}{Note}

\newcommand{\Z}{\mathbb{Z}}
\newcommand{\Q}{\mathbb{Q}}
\newcommand{\R}{\mathbb{R}}
\newcommand{\C}{\mathbb{C}}
\newcommand{\bb}{\mathbb}
\newcommand{\cali}{\mathcal}
\newcommand{\oo}{\omega}
\newcommand{\p}{\partial}
\newcommand{\Wk}{W^{k,p}(U)}
\newcommand{\W}{W^{1,p}(U)}
\newcommand{\g}{\mathfrak{g}}
\newcommand{\fk}{\mathfrak}
\newcommand{\h}{\mathfrak{h}}
\newcommand{\scr}

\begin{abstract}
We prove a twisting theorem for nodal classes in permutation-equivariant quantum $K$-theory, and combine it with existing theorems of Givental \cite{PermXI} to obtain a twisting a theorem for general characteristic classes of the virtual tangent bundle. Using this result, we develop complex cobordism-valued Gromov-Witten invariants defined via $K$-theory, and relate those invariants to $K$-theoretic ones via the quantization of suitable symplectic transformations. This procedure is a $K$-theoretic analogue of the quantum cobordism theory developed by Givental and Coates in \cite{QuantCob}. Using the universality of cobordism theory, we give an example of these results in the context of ``Hirzebruch $K$-theory", which is the cohomology theory determined by the Hirzebruch $\chi_{-y}$-genus.
\end{abstract}
\section{Introduction}
In the complex cobordism theory $\text{MU}^*(\cdot)$,
the Hirzebruch--Riemann--Roch formula
\[ \pi_*(A) = \int_{\textstyle [M]\cap \Td(T_{M})}\ChD(A)\ \in \ \QQ\otimes \text{MU}^*(pt) \]
expresses in cohomological terms
the push-forward along the map $\pi: M \to pt$ to the point
of a complex cobordism class $A\in \text{MU}^*(M)$ in a given (stably almost) complex manifold $M$. In this formula, $\ChD$ is the  
Chern-Dold character $\text{MU}^*(\cdot ) \to H^{\bullet}(\cdot; \QQ\otimes \text{MU}^*(pt))$, which is an isomorphism over $\QQ$,  while the \say{abstract Todd class}
\[ \Td(\cdot) = e^{ \sum_{k>0}\textstyle \ch_k(\cdot)} \]
is the universal multiplicative characteristic class of complex vector bundles, where the arbitrary coefficients $s_1,s_2,\dots$ form a certain set of free polynomial generators in the ring $\QQ\otimes \text{MU}^*(pt)$.
Consequently, one can interpret the cap-product $[M]\cap e^{\sum_{k>0} s_k \ch_k (T_{M})}$ as the cobordism-valued fundamental class of $M$. Henceforth we adopt the convention of \cite{Coatesthesis} and denote the rational version of cobordism theory by $U^*$. 

If instead of a manifold, we begin with a complex orbifold $\M$ with (virtual) tangent bundle $T\M$, the right hand side of the equation still makes sense, and can be used to define cobordism theoretic intersections. This leads to what is now known as {\em fake} cobordism-valued intersection theory. This point of view was adopted by Coates and Givental    
  \cite{QuantCob} in developing the theory of (albeit ``fake'' in our current terminology) cobordism-valued Gromov--Witten invariants, and expressing them in terms of cohomological ones (see \cite{Coatesthesis}, \cite{TwistedOrb}). In that theory, $[\M]$ is the {\em virtual} fundamental class of a moduli space $\M$ of stable maps to a given K\"ahler target space $X$.

 Complex cobordism theory reduces to complex K-theory when the abstract Todd class is specialized to its classical incarnation
 $\td(\cdot) :=\prod_{\text{Chern roots}\ x_i} x_i/(1-e^{-x_i})$. Applied to a holomorphic orbibundle $V$ on a complex orbifold $\M$, this leads to the fake holomorphic Euler characteristic $\chi^{fake}(\M;V):=\int_{M} \ch(V)\td(T_{\M})$.
 It is a rational number, which is only one summand (corresponding to $h=\id$ in the
 orbifold's isotropy groups) on the R.H.S. of the Kawasaki--Riemann--Roch formula
 \[ \chi(\M;V)=\chi^{fake}\left(I\M; \frac{\tr_h V|_{\M^h}}{\str_h \wedge^{\bullet}N^*_{\M^h}}\right) .\]
   The latter expresses the {\em true} (and integer)
   holomorphic Euler characteristic $\chi(\M;V):=\sdim H^{\bullet}(\M;V)$ of the orbibundle in cohomological terms of the inertia orbifold $I\M$. \par
   
   Using this as a starting point, one can define true quantum K-theory (as opposed to the fake one), i.e. the theory of holomorphic Euler characteristics of holomorphic orbibundles on the moduli spaces of stable maps. It is based on the notion of the virtual structure sheaf introduced by Y.-P. Lee \cite{yplee}.\par 
    
For manifolds, there is a similar relation between cobordism theory and $K$-theory, analogous to the one between cobordism theory and cohomology, it works as follows: Given a compact complex manifold  $M$,
to every integer polynomial $P$ in $\dim M$ variables one can associate 
the $P$-twisted (virtual) structure sheaf
\[  \O_P:=\O \otimes P(\wedge^1T_{M},\wedge^2T_{M},...,\wedge^{\dim M}T_{M}),\] 
and respectively define true holomorphic Euler characteristics \[ \chi_P(M;V):=\chi(M;V\otimes O^{vir}_P)\]
of vector bundles $V\in K^0(M)$.  
In a similar manner to the cohomological case, these integers can be interpreted as cobordism-valued intersection
numbers.  \par
Indeed, taking in the role of $P$ the Newton polynomials $N_r$ expressed as polynomials of elementary symmetric functions, we obtain the Adams
operations, $\Psi^r(T_{\M})$. Over the rationals, the general multiplicative K-valued characteristic class of complex vector bundles has the form
$e^{\sum_{r>0} S_r \Psi^r(\cdot)/r}$, where the arbitrary coefficients $S_1,S_2,\dots$ can be considered as certain independent elements in a completion of the coefficient ring
of cobordism theory.\par

The analogue of the Chern-Dold character is denoted $\Ch_K$, which is an isomorphism:
\[ \Ch_K: U^{*}(M) \stackrel{\ChD}{\longrightarrow} H{ev}(M;U^*(pt)) \stackrel{\ch^{-1}}{\longrightarrow} K^0(M)\otimes U^*(pt)\]
 Thus, we can define
\[ U^*(M)\ni A \mapsto \pi_*(A):=\chi \left(M; \Ch_K(A)\otimes 
e^{\sum_{r>0}S_r\Psi^r(T_{M})/r}\right).\] 

The right hand side of this formalism makes sense in the context of orbifold $K$-theory as well. So we can use $K$-theory to emulate a version of cobordism-theoretic intersection theory for an orbifold, by using the ring $K^0(\M)\otimes^{hat}U^*(pt)$, with pushforward given by $\pi_*\alpha=\chi(X;\alpha\otimes e^{\sum_{r>0} S_r\Psi^r(T_{\M}/r)}$.   \par

We call  the brand of cobordism
theory thus obtained {\em multiplicative}, and refer to $Ch_K$ and $Td_K$ as the \emph{multiplicative Chern-Dold character} and \emph{multiplicative Todd class} respectively. It is not genuinely \say{true} cobordism-valued intersection theory on $\M$, as we are unclear what that should mean, but being defined by means of the true (as opposed to fake) K-theory on $\M$, it is \say{less fake} than that of
Coates--Givental.\par

\begin{rem}
The name \say{multiplicative} comes from the relationship between $Ch_K$ and the formal group law determined by $K$-theory, which is that of the multiplicative group. This relationship will be explained in further detail in Section 4.1.
\end{rem}

In what follows, we develop multiplicative quantum cobordism theory\footnote{We are as yet unsure of the relationship between the quantum deformation of cobordism-theory obtained by incorporating our invariants into the product and the deformation recently introduced by Buchstaber-Veselov in \cite{BuchVes}.}  (i.e. apply this construction to the moduli spaces of stable maps), and express thus defined cobordism-valued Gromov-Witten invariants in terms of $K$-theoretic ones.

The latter task reduces to computing K-theoretic Gromov-Witten invariants based on the twisted structure sheaves
$\O^{vir}_P$ in terms of those with $P=1$. \par For this, we need three types
of \say{twisting} results of quantum K-theory. Two of them are already contained
in \cite{PermXI}, and the third one is proved in Section 7 below. The proofs of these theorems rely on modifying the quantum adelic Hirzebruch-Riemann-Roch formula due to Givental \cite{PermIX}. This formula relies on the more general framework of permutation-equivariant quantum $K$-theory, and expresses true $K$-theoretic invariants in terms of fake ones. \par

The universal nature of cobordism theory means that the results of this paper can be specialized to other cohomology theories, as a particular example we consider \say{Hirzebruch $K$-theory}, the theory whose pushforward map is based on the Hirzebruch $\chi_{-y}$ genus. \par

\section{Permutation-Equivariant \texorpdfstring{$K$}{K}-theoretic invariants}
We recall the definition of permutation-equivariant $K$-theoretic Gromov-Witten invariants and the associated potentials, using the definition introduced in \cite{PermIX}.\par
Given $h\in S_n$ with $\ell_r(h)$ cycles of length $r$, with $r$ ranging from 1 to $s$, $h$ acts on $X_{g,n,d}$ by permuting the marked points. \par
For each $r$, given inputs $w_{r1},\dots,w_{r\ell_r}$ each of the form $\sum \phi_mq^m$, for $\phi_m\in K^0(X)\otimes \Lambda$, associate to the input $w_{rk}$ the element $W_{rk}\in K^0(X_{g,n,d}):=\prod_{\alpha=1}^r\sum_m \text{ev}_{\sigma_\alpha}^*\phi_mL_{\sigma_\alpha}^m$, where $\sigma_\alpha$ are the marked points permuted by the $k$th cycle of length $r$, and $L_{\sigma_\alpha}$ are the corresponding cotangent line bundles on $X_{g,n,d}$\par

Given a partition $\ell$, a genus $g$, and a degree $d$, $S_n$-equivariant correlators are defined as follows $$\langle w_{11},\dots,w_{1\ell_1},\dots\rangle_{g,\ell,d}:=\prod_r r^{-\ell_r}str_hH^*(X_{g,n,d},\mathcal{O}^{vir}_{g,n,d}\prod_{i=1}^s\otimes_{j=1}^{\ell_i} W_{ij}).$$\par

The elements of $\Lambda$ act $\Psi-$linearly, i.e. scaling the $r$th input by $s\in \Lambda$ is equivalent to multiplication by $\Psi^r(s)$. \par

Define the genus $g$ potential function $\mathcal{F}^g_X$ and total descendant potential $\mathcal{D}_X$ are defined as follows:
$$\mathcal{F}^g_X:=\sum_dQ^d\sum_{\ell}\frac{1}{\prod_r\ell_r!}\langle\dots,t_i,\dots,\rangle_{g,\ell,d}.$$
$$\mathcal{D}_X:=e^{\sum_{g\geq 0}(
\sum_{r>0 } \hbar^{k(g-1)}\frac{\Psi^r}{k}(R_r(F_g)))}.$$
The variables $t_r$ are the same for all inputs coming from cycles of length $r$, the operator $R_r$ takes $F(t_1,t_2,\dots)$ to $F(t_r,t_{2r},\dots)$, and $\Psi^r(\hbar)=\hbar^r$.\par
After a dilaton shift of $1-q$ in each input, $\mathcal{D}_X$ defines a quantum state in the symplectic loop space $\mathcal{K}^\infty$, which is given as a $K$-module by $\prod_{r\in \Z_+} \mathcal{K}$, equipped with the symplectic form $\Omega^\infty(f,g)=\bigoplus \frac{\Psi^r}{r}\Omega(f_r,g_r)$.\par
The positive and negative spaces $\mathcal{K}^\infty_+$ and $\mathcal{K}^\infty_-$ are inherited from $\mathcal{K}$.\par

The ordinary genus-$g$ and descendant potentials $\mathcal{F}^{g,K}_X$ and $\mathcal{D}_X^K$ discussed in the introduction are recovered from the $S_n$-equivariant ones by letting $t_r=0$ for $r>1$. Concretely this means the following:

Ordinary $K$-theoretic correlators draw inputs from the algebra $K[q^\pm]$, and are given by the formula 
$$\langle \alpha_1,\dots,\alpha_n\rangle_{g,n,d}=\chi(X_{g,n,d}; \mathcal{O}^{vir}\cdot \prod_{i=1}^n ev_i^*\alpha_i(L_i)).$$

The genus-$g$ and total descendant potentials are defined by
$$\mathcal{F}^{g,K}_X(\mathbf{t})=\sum_{d\in H^2(X),n\geq 0} \frac{Q^d}{n}\langle \mathbf{t},\mathbf{t},\dots,\mathbf{t}\rangle_{g,n,d}$$
$$\mathcal{D}_X^K=e^{\sum_g \hbar^{g-1}\mathcal{F}^{g,K}_X}.$$

\subsection{Symplectic Loop Spaces and Quantization}

We now introduce the vector spaces that will appear in our formalism: $\mathcal{K_r}$ consists of rational functions in $q$ with coefficients in $K$ and poles only at $0,\infty$ and roots of unity. The symplectic form is $$\Omega^r(f,g):=-\text{Res}_{0,\infty} (f(q),g(q^{-1}))^{(r)}\frac{dq}{q}.$$ Here $(,)^{(r)}$ denotes the $K-$theoretic Poincare pairing twisted by the operation $\Psi^(r)$.\par 

The polarization is described as follows: $\mathcal{K^r}_+$ consists of Laurent polynomials in $q$ and represents inputs to $\mathcal{D}_X$ comign from cycles of length $r$, and $\mathcal{K}_-$ is $\{ f: f(0)\neq \infty, f(\infty)=0 \}$. We also modify the quantization formulas by replacing $\hbar$ with $\hbar^r$.\par

Define $$\mathcal{K}^\infty:=\prod_{r>0}^\infty \mathcal{K}^r$$, with symplectic form $\Omega(f,g)=\sum_r \Omega^r(f_r,g_r)$, and polarizations inherited from $\mathcal{K}^r$. 

$\mathcal{D}_X$ is a function on $\mathcal{K}^\infty_+$, and thus defines a quantum state $\langle \mathcal{D}_X\rangle$. To make certain equations homogenous, we impose that this construction is done after applying the \emph{dilaton shift}, a translation replacing $t_r$ with $t_r-(1-q)$.\par

Since specializing to ordinary $K$-theoretic invariants is equivalent to setting $t_r=0$ for $r>1$, $\mathcal{D}_X^K$ naturally defines a quantum state on $\mathcal{K}:=\mathcal{K}^1$. 

\section{Twisting theorems}
We can define twisted $K$-theoretic invariants by tensoring $\mathcal{O}^{vir}$ with other classes from $K^0(X_{g,n,d})$. 

These classes take the form $\mathbf{E}=e^{\sum \frac{\Psi^k}{k}(E_k)}$. Where the $E_k$s must all be of the following 3 types:
\begin{itemize}
    \item Type I: $E_k=\text{ft}_*\text{ev}_{n+1}^*V_k$, where $V_k\in K^0(X)$. 
    \item Type II:$E_k=\text{ft}_*\text{ev}_{n+1}^*F_k(L_{n+1})$, where $L_{n+1}$ is the universal cotangent line and $F_k$ is a Laurent polynomial with coefficients in $K^0(X)$, with $F_k(1)=0$. If $F_k=\sum a_kq^k$ then $\text{ev}_{n+1}^*F_k(L_{n+1})$ is shorthand for $\sum \sum_k \text{ev}_{n+1}^*a_k L_{n+1}$, these are $K$-theoretic versions of the $\kappa$-classes introduced by Kabanov and Kimura in \cite{kabanov}.
    \item Type III: $E_k=\text{ft}_*\text{i}_*\text{ev}_{n+1}^*F_k(L_+,L_-)$, where $\text{i}:\mathcal{Z}\to U_{g,n,d}$ is the inclusion of the codimension-2 locus of nodes, and $F_k$ is a symmetric Laurent polynomial in two variables with coefficients as above. $L_+$ and $L_-$ denote cotangent line bundles to the branches at the node. Note we could equivalently write $E_k=\text{ft}_*(ev_{n+1}^*F_k(L_+,L_-)\otimes \mathcal{\mathcal{O}_Z})$. 
\end{itemize}

Henceforth we will omit the subscript $n+1$ from the evaluation map when there is no ambiguity. 

The construction of the twisted total descendant potential $\mathcal{D}_X^{E}$ is identical to its untwisted counterpart. except that the operator $R_r$ also replaces $E_k$ with $E_{rk}$.  

The effect of twistings of any type on on' $\mathcal{D}_X$ is as follows: $\langle \mathcal{D}_X^{\mathbf{E}}\rangle=\nabla\langle \mathcal{D}_X\rangle$, where $\nabla$ is some operator on Fock space. The quantization formulas are also adjusted slightly, be replacing $\hbar$ with $\hbar^r$ in the $r$th component. Twistings can also be taken on top of each other, if $\mathbf{E}$ splits as $\mathbf{E_1}\mathbf{E_0}$, where $\mathbf{E_1}$ is a twisting of a particular type, and $\mathbf{E_0}$ is made up of other types, we also have $\langle \mathcal{D}_X^{\mathbf{E}}\rangle=\nabla\langle \mathcal{D}^{\mathbf{E_0}}_X\rangle$. 

The following theorems, proven in \cite{PermXI}, describes $\nabla$ for twistings of type 1 and 2 in terms of the symplectic geometry of $\mathcal{K}^\infty$ \par

\begin{thm}
For a twisting of type I, $\nabla$ is the quantization of the multiplication operator $f_r\mapsto \Phi_rf_r$ in the $r$th component, where $\Phi_r$ is the Euler-Maclaurin asymptotics of $e^{\sum_{k\neq 0} \frac{\frac{\Psi^{k}}{k}(E^{rk})}{1-q^k}}$. Here $\Phi$ is regarded as a $\Psi$-linear  symplectomorphism from $\mathcal{K}^\infty$ with symplectic form governed by the twisted Poincare pairing $(a,b)^{tw,r}=\chi(a\otimes b\otimes e^{\sum \frac{\Psi^{rk}}{k}(E^{rk})})$ to $\mathcal{K}^\infty$ with the standard symplectic form.
\end{thm}
\begin{thm}
For a twisting of type II, the operator $\nabla$ is  the translation on Fock space that changes the dilaton shift from $v_r=1-q$ to $v_r=(1-q)e^{\sum_k \frac{\Psi^{k}(F^{rk}(q)-F^{rk}(1))}{k(1-q^k)}})$. This can be interpreted as leaving the quantum states the same, but changing how they are obtained from the actual potentials.\par
\end{thm}

In this work, we prove an analogous theorem for twistings of type III:\par
\begin{thm}
For a twisting of type III, the operator $\nabla$ is of the form $e^{\frac{\hbar}{2}\sum_r\Delta_r}$, where $\Delta_r$ is an order-2 differential operator on determined by insertion of the symmetric tensor $$\frac{e^{\sum \frac{\Psi^{k}}{k}(F_{rk}(L_+,L_-))(1-L_+^kL_-^k)\phi^\alpha\otimes\phi_\alpha}-1}{1-L_+L_-}\in \mathcal{K}_+\otimes \mathcal{K}_+$$ into $t_r$.

The result is equivalent to changing the negative space of the polarization on $\mathcal{K}^\infty$. 
\end{thm}\par

We take a moment to describe concretely the change of polarization in Theorem 3.3. The operator is the quantization of the time-1 flow of the quadratic Hamiltonian given in standard Darboux coordinates by $(p,Sp)$, which results in the symplectomorphism given in Darboux coordinates as $(p,q)\mapsto (p,q+Sp)$. \par
Phrased invariantly, $S$ is the map $\mathcal{K}_-\to \mathcal{K}_+$ given in the $r$th coordinate by dualizing the symmetric tensor$\frac{e^{\sum \frac{\Psi^{k}}{k}(F_{rk}(L_+,L_-))(1-L_+^kL_-^k)\phi^\alpha\otimes\phi_\alpha}-1}{1-L_+L_-}$ with respect to the symplectic form. It induces a change of negative space from $q=0$ to $q=Sp$, but leaves $\mathcal{K}_+$ constant. 
The twisted potential represents the same quantum state as the untwisted one, but in the Fock space on $\mathcal{K}^\infty$ constructed with respect to a different negative polarization. \par
If one interprets $\frac{1}{1-L_+L_-}$ as a tensor in $\mathcal{K}^\infty_+\otimes \mathcal{K}^\infty_-$, then dualizing gives the identity map from $\mathcal{K}^\infty_-$ to $\mathcal{K}^\infty_-$. Renaming $L_+$ to $q$ and $L_-$ to $x$, this map is given by $f(q)\mapsto -\Res_{0,\infty} \frac{f(q)}{1-q^{-1}x}\frac{dq}{q}$, provided $x$ is interpreted as being close to infinity. \par
Adding this to the map $\mathcal{K}^\infty_-\to \mathcal{K}^\infty_+$ determined from the operator gives the map $\mathcal{K}^\infty_-\to \mathcal{K}^\infty$ sending $(p,0)$ to $(p,Sp)$. In summary, the new polarization is determined in the $r$th coordinate by the image of $\mathcal{K}_-$ under the map obtained by dualizing the expression \par
$$\frac{e^{\sum \frac{\Psi^{l}}{l}(F_{rk}(L_+,L_-))(1-L_+^kL_-^k)\phi^\alpha\otimes\phi_\alpha}}{1-L_+L_-}$$\par

We will henceforth use expressions of this kind to label polarizations.
\section{Multiplicative quantum cobordism theory}

\subsection{Complex-oriented cohomology theories and formal group laws}
A complex-oriented cohomology theory is a generalized cohomology theory which admits Chern classes for complex vector bundles. For a theory $A^*$, a complex orientation is determined its value on the universal line bundle, which is an element $u_A$ in $A^2(\C P^\infty)$. For the standard orientation of cohomology, this element is traditionally denoted $z$. In $K$-theory, we use $1-q^{-1}$, where $q$ is the class of the universal line bundle itself.\par

A complex-oriented cohomology theory defines a formal group law by the rule $$c_1(L_1\otimes L_2)=F(c_1(L_1),c_1(L_2)).$$ For ordinary cohomology theory, it is additive formal group law, for $K$-theory, the result is the multiplicative formal group law since $c_1^K(L_1\otimes L_2)=1-L_1^{-1}L_2^{-1}$.\par

Complex cobordism theory is the cohomology theory defined by the Thom spectrum. It admits a tautological orientation $u$ coming from the isomorphism between the Thom space of the universal line bundle and $\C P^\infty$. This orientation is universal in the following sense: a choice of complex orientation on a homotopy commutative ring spectrum $A$ corresponds to a map $\phi:MU\to A$. Similarly, the formal group law associated with $\text{MU}^*$ is the universal one, meaning that the coefficients of the defining power series are free generators of $\text{MU}^*(pt)$, which is the ring of manifolds under complex cobordism. Over $\mathbb{Q}$, it a polynomial ring generated by $\C P^{k}$ in degree $-2k$. \par
 The Chern-Dold character mentioned in the introduction is the isomorphism $U^*(X)\to H^*(X,U^*(pt))$ determined by sending $u\in U^2(\C P^\infty)$ to $u(z)\in H^*(\C P^\infty, U^*(pt))$, where $u(z)$ is the exponential of the cobordism-theoretic formal group law. Specializing to $K$-theory gives the series $1-e^{-z}$, which is indeed an isomorphism from the additive to the multiplicative formal group. \par 
 The multiplicative Chern character $\Ch_K: U^*(X)\to K^0(X)\otimes U^*(pt)$ is defined by $Ch\circ ch^{-1}$, and it is determined by the image $u\in U^2(\C P^n)$, which is a power series in $1-q^{-1}$ which we denote by $u(1-q^{-1})$. \par

As mentioned in the introduction $\Ch_K$ and the cobordism-theoretic pushforward map satisfy a Hirzebruch-Riemann-Roch formula, i.e. for $\pi$ the map $X\to pt$, we have: $$\pi_*\alpha=\chi(X;Ch_K(\alpha)\Td_K(TX)),$$ 
where the \emph{multiplicative Todd class} $\Td_K$ is the universal stable exponential characteristic class in $K$-theory. It is defined on the universal line bundle by the formula $\frac{1-q^{-1}}{u(1-q^{-1})}$. \par
This is a consequence of a more general theorem of Dyer that gives a similar result between any two cohomology theories \cite{Dyer}, but can be viewed more concretely as combination of the Hirzebruch-Riemann-Roch formula for $U^*$ and the usual one relating pushforwards in $K$-theory and cohomology.\par

The logarithm of the formal group law of cobordism theory is given by Mischenko's formula as $$z(u)=u+\sum_{n\geq 1} [\C P^{n}]\frac{u^{n+1}}{n+1}.$$\par

One can thus explicitly compute $u(1-q^{-1})$ as the series inverse of $1-e{-z(u)}$. There are some generators $b_k$ of $U^*(pt)$ such that: $u(1-q^{-1})=1-q^{-1}+\sum_{k\geq 1} b_k (1-q^{-1})^{k+1}$. \par

Similarly: $ln(\frac{1-q^{-1}}{u(1-q^{-1})})=\sum_{k\geq 1} a_k(1-q^{-1})^k$ for a different set of generators $a_k$. After completing with respect to this grading, $\sum_{k\geq 1} a_k(1-q^{-1})^k$ can be rewritten as a series in $q^{-k}$, denoted $s(q)=\sum_{k\geq 0} c_kq^{-k}$. The $c_{k\geq 1}$ are independent in the completion of $U^*(pt)$, but $c_0$ is determined by the requirement $s(1)=0$.\\

Since $\Td_K$ is multiplicative and Adams operations are additive, the formula for the multiplicative Todd class of a general bundle is:
$$\Td_K(\cdot)=e^{\sum_{k\geq 0} \frac{c_k}{k}\Psi^k(\cdot)}.$$
Here $\Psi^0$ is the rank operator. This is the universal $K$-theoretic characteristic class mentioned in the introduction, with the additional requirement of stability. \par

The stability requirement can be relaxed in the following way. Given a characteristic class $C$ with $C(1)=t$ for $t$ some unit, we can regard it as coming from a series $\frac{1-q^{-1}}{u(1-q^t)}$, where $u(1-q^t)$ is a homomorphism the multiplicative group with orientation given by $(1-q^{-t})$ instead of $(1-q^{-1})$. This scales the logarithm $z(u)$ by a factor of $\frac{1}{t}$. \par

Using $C$ and $u(1-q^{-t})$ define new versions of $\Ch_K$ and $\Td_K$, however, the resulting pushforwards have the same value as if we used the normalized version of $C$ instead. To see this, apply the ordinary Riemann-Roch formula to rewrite $\chi(X; Ch_K(\alpha)\Td_K(TX))$ as an integral over $X$. The $\frac{1}{t}$ coefficients appearing from the expansion of $\td(\Td_K(TX))$ and $\ch(\Ch_K(\alpha))$ cancel in the top degree. The same result is true in the orbifold setting, which can be shown by applying Kawasaki-Riemann-Roch and then considering top-degree terms on each stratum.\par 

Keeping this in mind, the $c_0$ term in the exponential expression for $Td_{K}$ can be ignored, provided we apply the above modifications consistently.  

\subsection{Cobordism-valued Gromov-Witten invariants}
Define the algebra $U$ to be $\widehat{U}^*(X)$, where the hat denotes completion by the grading introduced in the previous section, and further completion to ensure $u(1-q^{-1})$ is a Laurent polynomial in $q$ (the latter may involve adding an additional variable). \par

Define $q(u)$ to be $e^{z(u)}$. The inputs to cobordism-theoretic correlators are drawn from $U[q(u)^{\pm}]$, regarded as a subalgebra of $U(u)$. This algebra contains $u$ as well as $u^*:=u(1-q(u))$, which represents the first $U^*$-theoretic Chern class of the dual to the universal line bundle.\par

For $\alpha_i\in U[q(u)^{\pm}]$, the cobordism theoretic correlators are defined via the right hand side of the Hirzebruch-Riemann-Roch formula, i.e.  $$\langle \alpha_1,\dots,\alpha_n\rangle^U_{g,n,d}=\chi(X_{g,n,d};\mathcal{O}^{vir}\cdot \Td_K(\mathcal{T}^{vir}) \prod_{i=1}^n ev_i^*\Ch_K\alpha_i(L_i)).$$ \par

The genus $g$ and total descendant potentials $\mathcal{F}^{g.U}_X$ and $\mathcal{D}_X^U$ are defined in the same way as for $K$-theory. \par

\subsection{The loop space \texorpdfstring{$\mathcal{U}$}{U}}
We construct the space $\mathcal{U}$ in a similar manner to $\mathcal{K}$. As a $U$-module, $\mathcal{U}$ is defined as $U[q(u)^\pm]$ localized at $1-q(u)^m$ for each $m\in \Z_{\neq 0}$. The symplectic form is $$\Omega^U(f,g):=Res_{q(u)=0,\infty} (f(u),g(u^*))^{U}dz(u).$$ $(,)^U$ denotes the cobordism-theoretic Poincare pairing. \par

As with $\mathcal{K}$, $\mathcal{U}_+$ is $U[q(u)^\pm]$, however the negative space is not the natural analogue of $\mathcal{K}_-$, consisting of functions holomorphic at 0 and vanishing at $\infty$. Rather, it is obtained from that space by dualizing the symmetric tensor $\frac{1}{c_1^U(L_1^*\otimes L_2^*)}$, and taking the image under the resulting linear map. With the above data, $\mathcal{D}_X^U$ defines a quantum state $\langle \mathcal{D}_X^U\rangle$ of $\mathcal{U}$ after a dilaton shift of $u^*$.\par

\section{Formula for \texorpdfstring{$\mathcal{D}_X^U$}{}}
We are now in the position to state the formula relating $\mathcal{D}_X^K$ and $\mathcal{D}_X^U$:

The \emph{quantum multiplicative Chern character} $\qCh_K$, defined by extending $\Ch_K$ by $u\mapsto u(1-q^{-1})$, (equivalently $q(u)\mapsto q$), is a linear isomorphism from $\mathcal{U}$ to $\mathcal{K}$ (provided $\Lambda$ is chosen to be $u^*(pt)$ completed appropriately). $\qCh_K$ is not a symplectomorphism, since it transforms the cobordism-theoretic Poincare pairing into the $K$-theoretic pairing with the insertion of $t\cdot \Td_K(TX)$. Furthermore, it does not identify dilaton shifts nor polarizations. Roughly, after correcting these discrepancies, $\qCh_K$ identifies the quantum states. More precisely, the following formula holds:

\begin{thm}
$$\qCh_K\langle D_X^U\rangle=\nabla\langle D_X^K\rangle$$
Where $\nabla$ consists of 3 operators:

\begin{itemize}
    \item The quantization of the scalar multiplication by the asymptotic expansion of $\Td_K(\frac{TX-1}{1-q})$, which is regarded as a symplectomorphism from $\mathcal{K}$ with symplectic structure twisted by $t\cdot \Td_K(TX-1)$ to $\mathcal{K}$ with its original symplectic structure. Thus viewed, the quantization acts in the opposite direction. 
    
    \item A translation operator on Fock space which changes the dilaton shift to from $1-q$ to $\qCh_K(u^*)=u(1-q)$. 
    \item The quantization of a symplectomorphism of the form $(p,q)\mapsto (p,q+Sp)$, which leaves $\mathcal{K}_+$ unchanged and changes $\mathcal{K}_-$ into $\qCh_K(\mathcal{U}_-)$. 
\end{itemize}
\end{thm}

$\qCh_K$ identifies the potentials $\mathcal{D}_X^U$ and $\mathcal{D}_X^{K,tw}$, where the twisting class is $\Td_K^U(\mathcal{T}^{vir})$. We recall the decomposition of $\mathcal{T}^{vir}$ in $K^0(X_{g,n,d})$ proved in \cite{Coatesthesis}:  $$\mathcal{T}^{vir}=-\text{ft}_*(L^{-1}-1)+\text{ft}_*(\text{ev}^*(T_X-1))-\text{ft}_*i_*\mathcal{\mathcal{O}_Z}^*.$$ Thus twisting by $\Td_K(\mathcal{T}^{vir})$ induces one twisting of each type.\par 

By theorems 2.1, 2.2, and 2.3, $\langle \mathcal{D}_X^{tw} \rangle=\nabla' \langle \mathcal{D}_X\langle$, where  $\nabla'$ is an operator that encodes a change of symplectic form, dilaton shift, and polarization.\par 

So the formula is equivalent to showing that $\nabla=\nabla'$, and that $\qCh_K$ is a symplectomorphism, which respects dilaton shift and polarization, provided that the symplectic structure on $\mathcal{K}$ is the one determined by $\nabla$. \par
Twisting by $\Td_K(\mathcal{T}^{vir})=e^{\sum_{k>0} \frac{\Psi^k}{k}(s_k\mathcal{T}^{vir})}$, for some particular choices of $s_k$, results in three twistings, one of each type:
\begin{itemize}
    \item Type I: $e^{\sum_{k<0} \frac{\Psi^k}{k}(s_k\text{ft}_*(1-L^{-1}))}$
    \item Type II: $e^{\sum_{k<0} \frac{\Psi^k}{k}(s_k\text{ft}_*ev^*(TX-1))}$
    \item  Type III: $e^{\sum_{k<0}\frac{\Psi^k}{k}(s_k(-\text{ft}_*i_*\mathcal{\mathcal{O}_Z})^*}$. 
\end{itemize} 
These result in the following changes:\par
\begin{itemize}

    \item Multiplication operator and symplectic pairing: Since the twisting of type I is $\Td_K(\text{ft}_*\text{ev}^*(T_X-1))$, the resulting multiplication operator is equivalent to changing the Poincare pairing into the following: $$\frac{1}{\Td_K(1)}\Res_{q=0,\infty}\chi(X;f(q)g(q^{-1}) \Td_K(T_X))\frac{dq}{q}.$$ The operator itself is the asymptotic expansion of $\frac{\prod_{m\leq 0}\Td_K(\alpha_iq^m)}{\Td_K(q^m)}$. The residue operations on $\mathcal{U}$ and $\mathcal{K}$ themselves coincide since $\qCh_K(dz(u))=\frac{1}{\Td_K(1)}d\log(q)=\frac{1}{\Td_K(1)}\frac{dq}{q}$. 
    
    \item Dilaton shift: The dilaton shift changes to $(1-q)e^{\sum_{k<0}\frac{\Psi^k}{k}(s_k(1-q))}$, which is the asymptotic expansion of $(1-q)\Td_K(q^{-1})=u(1-q)=\qCh_K(u^*)$.

    \item Change in polarization: The twisting of type III is by $\Td_K(-\text{ft}_*\text{i}_*\mathcal{\mathcal{O}_Z}^*)=1/\Td_K^*(\text{ft}_*i_*\mathcal{\mathcal{O}_Z})$.
    Here $\Td_K^*(V)$ denotes $\Td_K(V^*)$. So the expression determining the new polarization is:
    $$\frac{1}{(\Td_K^*(1-L_+L_-)(1-L_+L_-)}=\frac{\Td_K(L_+L_-^*)}{\Td_K(1)(1-L+L_-)}=\qCh_K(\frac{1}{c_1^U(L_+^*L_-^*)})$$

\end{itemize}

\section{Specialization and examples}
\subsection{Other cohomology theories}
Over $\mathbb{Q}$, the universality of cobordism theory also holds for cohomology rings. Given a cohomology theory $A$, the specialization map $\phi: U^*(pt)\to A^*(pt)$ is given by sending $[\C P^n]$ to the pushforward to the point of the class $1\in A^*(\C P^n)$. One recovers $A^*(X)$ by restriction of scalars from $U^*(X)$, which is exact over $\mathbb{Q}$. \par 

In this way, one can in principle specialize the constructions of the previous section to any complex oriented cohomology theory, and thus define Gromov-Witten invariants valued in that theory. However, since we use a completed version of $U^*$, the map $\qCh_K$ and the class $\Td_K(\mathcal{T}^{vir})$ will only be well-defined if $\phi$ factors through the completion, i.e. if $u_A(1-q^{-1})$ is actually a Laurent polynomial in $q$. \par 
We can also use the same framework to define invariants for algebraically-oriented theories. Levine and Morel's theory of algebraic cobordism outlined in \cite{Levine} has the same universality properties among algebraic theories as $\text{MU}^*$ does for complex oriented ones. The necessary Riemann-Roch theorems are due to Smirnov (\cite{Smirnov}). So the formalism we have constructed works equally well in this context. \par
We can also in principle extend multiplicative cobordism theory can also to permutation-equivariant invariants by replacing the holomorphic Euler characteristics used to define the correlators with supertraces, and constructing the resulting potentials analogously to the $K$-theoretic case. However the twisting class in the $r$th coordinate becomes $e^{\sum s_{rk}\frac{\Psi^k}{k}}$, so each input $t_r$ would have to be interpreted as coming from a different specialization of cobordism theory. If we are working in a specialization where $s_{rk}=\Psi^r(s_k)$, this is not an issue. 

\subsection{Example: the \texorpdfstring{$\chi_{-y}$}{}-genus and Hirzebruch \texorpdfstring{$K$}{K}-theory}
The Hirzebruch $\chi_{-y}$-genus is a polynomial deformation of the holomorphic Euler characteristic, it is defined on complex manifolds by $$\chi_{-y}(C)=\chi(M;\sum_{p}(-y)^p\Omega_X^p),$$ with a $y$ a formal variable.\par

The same definition extends to virtually smooth orbifolds $\mathcal{X}$ if we interpret $\Omega_\mathcal{X}$ to the the virtual cotangent bundle. The class $C_y(V):=\sum_p(-y)^p(V^*)^p)$ can be rewritten as $e^{\sum_{k\leq 0} \frac{1}{y^k}\frac{\Psi^k}{k}(V)^k}$. We can treat this as if it were a multiplicative Todd class, but from $K$-theory to itself, with a different choice of complex orientation, we call this modification of $K$-theory \emph{Hirzebruch $K$-theory}. \par 

We can thus use the class $C_y(\mathcal{T}^{vir})$ to define Gromov-Witten invariants, which in generic situations compute (virtual) $\chi_{-y}-$genera of suborbifolds $X_{g,n,d}$ representing stable maps subject to certain restrictions. If $L$ is a line bundle then $C_y(L)=1-yL^*$, and $C_y(1)=1-y$. Hence $C_y$ comes from the morphism of formal group laws given by the series $u(1-q^{-1})=\frac{1-q^{-1}}{1-yq^{-1}}$. If we complete the base algebra with respect to $y$, $u(1-q^{-1})$ becomes a Laurent polynomial in $q$. We refer to $K$-theory twisted this way as \emph{Hirzebruch $K$-theory}. \par
If we let Adams operations act on $y$ as $\Psi^k(y)=y^k$, then $C_y(V)$ is the $S^1$-equivariant $K$-theoretic Euler class of $V$ with equivariant parameter $y$. In this guise these invariants appear in the study of the quantum $K$-theory of Grassmanians, as in \cite{Xiaohan}. Note that this also means these invariants extend naturally to the permutation-equivariant case. \par

Using the formula from section 5, the transition to Hirzebruch $K$-theory has the following effects on the symplectic loop space: 
\begin{itemize}
\item The multiplication operator from the type I part of the twisting changes the Poincare pairing to $(a,b)=\chi(X;a\cdot b\cdot \Td_y(TX))$, and scales the symplectic form by $\Td_y(1)=\frac{1}{1-y}$. 
\item The dilaton shift becomes $u(1-q)=\frac{1-q}{1-yq}$, demonstrating formal group inversion.
\item The subsequent polarization changes to the one determined by $\frac{1}{1-y}\frac{1-yL_+L+_-}{1-L_+L_-}$. The map $f\mapsto -\Res_{0,\infty}\frac{1}{1-y}\frac{f(q)(1-yq^{-1}x)}{q-x}\frac{dq}{q}$ sends $f\in \mathcal{K}^\infty_-$ to $f+\frac{y}{1-y}f(0)$, so the new negative space is $\{f: f(\infty)=yf(0)\neq \infty\}$.
\end{itemize}

For genuinely smooth orbifolds, when $y=1$, this version of the $\chi_{-y}$-genus becomes the ordinary topological Euler characteristic. So in cases where $X_{g,n,d}$ is genuinely smooth, which includes cases where $X$ is homogeneous, in particular $\bar{M}_{g,n}$. Applying the corresponding approach using the cohomologically defined invariants of \cite{QuantCob} instead yields the orbifold Euler characteristic, which is a rational number given by a weighted count of simplices. This illustrates the general principle that multiplicative cobordism-theoretic invariants will have different relationships to the orbifold structure of $X_{g,n,d}$ than ``fake'' ones.
\par
However, the symplectic formalism degenerates in this limit, so any computations must be done for a general $y$, and then specialized. We postpone a detailed discussion of the kinds of invariants that thus occur to another work.

\section{Proof of theorem 3.3}

\subsection{Adelic Formula for \texorpdfstring{$\mathcal{D}_X$}{DX}}
We recall the adelic formula for $\mathcal{D}_X$, which recasts the $K$-theoretic potential into purely cohomological terms. The proof of theorem 2.2 will rely heavily on this formula.\par
We define the adelic symplectic loop space $\underline{\mathcal{K}}^\infty=\bigoplus_{M\in \Z_+} \mathcal{K}^{fake}(X\times B\Z_M)$, where $\mathcal{K}^{fake}(X\times B\Z_M)$ denotes the loop space of the fake quantum $K$-theory of the orbifold $X\times B\Z_M$. Each summand splits as a direct sum of $M$ sectors $\mathcal{K}_M^\zeta$ labelled by roots of unity $\zeta$, each isomorphic to $K((q-1))$. \par
The symplectic structure on $\mathcal{K}^{fake}(X\times B\Z_M)$ comes from an additional twisting of fake quantum $K$-theory which we outline later, and is described as follows: The symplectic form $\Omega^{tw}$ pairs $\mathcal{K}_M^\zeta$ with $\mathcal{K}_M^{\zeta^{-1}}$ by $\Omega^{tw}(f,g)=\frac{1}{M}(f(q),g(q^{-1}))^{(r)}$, where $(\Psi^ra,\Psi^rb)^{(r)}=r\Psi^r(a,b)$, for $(a,b)$ the usual Poincare pairing, and $r$ is the index of $\langle\zeta\rangle$ in $\Z_M$. Let $m(\zeta)=\frac{M}{r(\zeta)}$ denote the primitive order of $\zeta$.\par
Define the adelic potential $\underline{\mathcal{D}}_X$ to be $\bigotimes_M\mathcal{D}^{tw}_{X\times B\Z_M}$. \par
We can resum the component spaces according to $r$, to describe the adelic space as:$$\bigoplus_{\zeta}\bigoplus_r \mathcal{K}_r^\zeta$$ Where the first sum is taken over all roots of unity.\par
After resumming, the symplectic form becomes $$\underline{\Omega}^\infty(f,g)=\sum_{\zeta}\frac{1}{m(\zeta)}\sum_rRes_{q=1}(f_r^\zeta(q^{-1}),g_r^{\zeta^{-1}}(q))^{(r)}\frac{dq}{q}.$$\par
The adelic map $\Phi: (f_1,\dots)\mapsto \Psi^r(f_r(\frac{q^{\frac{1}{m}}}{\zeta}))$ defines a symplectic $\Psi$-linear transformation between $\mathcal{K}^\infty$ and $\underline{\mathcal{K}}^\infty$, which respects positive, but not negative polarizations, since an element of $\mathcal{K}^\infty_-$ will not be polar at every root of unity. \par
A result of \cite{PermIX} is that $\langle \mathcal{D}_X\rangle=\Phi^*e^{\hbar/2\sum_{r,\zeta\eta\neq 1}\nabla_{r,\zeta,\eta}}\underline{\mathcal{D}}_X$, where $exp(\hbar/2\sum_{r,\zeta\eta\neq 1}\underline{\nabla}_{r,\zeta,\eta})$ is the quantization of the rotation changing the standard polarization on $\underline{\mathcal{K}}^\infty$ to the uniform polarization, which is determined by the image of $\mathcal{\mathcal{K}_-}^\infty$ under $\Phi$.\par
This formula has the form of Wick's summation over graphs, and arises from the application of the Lefschetz-Kawasaki-Riemann-Roch theorem to $\mathcal{D}_X$. The theorem states that for $\mathcal{X}$ a orbifold, $V$ an orbibundle, and $h$ a discrete automorphism of $\mathcal{X}$ that lifts to $V$:
$$str_h(\mathcal{X};V)=\chi^{fake}\big(\mathcal{IX}^h;\frac{tr_{\tilde{h}}(V)}{str_{\tilde{h}}N^*_{\mathcal{IX}^h|\mathcal{X}}}\big)$$ 
Here $\tilde{h}$ some lifting of $h$ on each component of $\mathcal{IX}^h$, and $\chi^{fake}(A;V)$ is defined to be $\int_{A} \ch(V)\td(TA)$, i.e. the pushforward in fake $K$-theory. This theorem is consequence of the usual Kawasaki-Riemann-Roch theorem, which was shown by Tonita in \cite{VirtKawasaki} to hold for virtually smooth orbifolds.\par
 We recall from \cite{PermIX} the following description of $\mathcal{I}X_{g,n,d}^h$:\par
The total space itself corresponds to a moduli space of stable maps from curves $\mathcal{C}$ with a symmetry $\tilde{h}$ accomplishing the permutation $h$ of marked points.\par
A connected component (henceforth referred to as a Kawasaki stratum) of this space is described by certain combinatorial data:\par
\begin{itemize}
    \item A graph $G$ dual to the quotient of the curve by the cyclic group generated by $\tilde{h}$.
    \item A positive integer $M_v$ for each vertex $v$ $M_v$ representing the order of $\tilde{h}$ on the vertex $v$. 
    \item The discrete characteristics (genus, degree) of the map on each irreducible component. 
    \item A labelling of the vertices of $G$ with eigenvalues of $\tilde{h}^r$ on the tangent lines to the branches at the ramification points of order $r$. These eigenvalues will be primitive $m$th roots of unity for $m=\frac{M_v}{r}$.
    \item A labeling of the edges of $G$ (corresponding to nodes) with pairs of eigenvalues of $\tilde{h}^r$ on each branch to the node. We require that these eigenvalues not be inverse to each other (i.e. the node is unbalanced), so the node cannot be smoothed within the stratum. 
\end{itemize}
After normalizing at the unbalanced nodes, each vertex represents a component of a Chen-Ruan moduli space of stable maps to the orbifold $X\times B\Z_M$. After doing this, the eigenvalue at a marked point or node also determines the sector of $\mathcal{I}(X\times B\Z_M)$ in which the evaluation map at that marked point lands. \par
Thus the KRR formula relates a correlator to some $fake$ $K$-theoretic correlators of $X\times B\Z_M$, which are additionally twisted by the denominator terms. These account for the twistings of fake $K$-theory that appear in the adelic space formalism. \par
Marrying the vertices at edges involves the application of a propagator operator for each edge, which coincides with the change of polarization from the standard to the uniform polarization. \par
\subsection{Twisted potentials}
The exact same argument applies essentially verbatim to twisted potentials, with two differences. The vertex potentials are further twisted by the restriction of the twisting class (we label the resulting potentials $\mathcal{D}^{tw,\mathbf{E}}_{X\times B\Z_M}$). And, only in the case of type III twistings, the edge operators are modified as well. \par

Our strategy will thus be to begin with the twisted potential $\mathcal{D}_X^{\mathbf{E}}$, where $\mathbf{E}$ denotes a twisting of type III, we pass to the adelic potential $\underline{\mathcal{D}_X^{E}}$, and analyze the vertex contributions coming from $E$ to relate $\underline{\mathcal{D}_X^{E}}$ and $\underline{\mathcal{D}_X}$. Then we use the adelic formula to convert that to a relationship between $\mathcal{D}_X^{E}$ and $\mathcal{D}_X$, which will involve comparing the respective edge operators. \par

Rather than beginning with $\mathcal{D}_X$, we could take $\mathbf{E}$ to be the composition of $\mathbf{E_0}$, a twisting of type I and II, and $\mathbf{E_1}$, a twisting of type $III$. The resulting argument would give a relationship between $\mathcal{D}_X^\mathbf{E}$ and $\mathcal{D}_X^{\mathbf{E_0}}$, and is identical to the case where $\mathbf{E_0}$ is trivial, so we just work in the latter setting to minimize notation.

\subsubsection{Vertex Contributions}
Let $\mathcal{\widehat{M}}$ be a Kawasaki stratum with ambient moduli space $X_{g,n,d}$ (from which the twisting classes are inherited). Let $\mathcal{C}$ be the universal curve, and $\widehat{\mathcal{C}}=\mathcal{C}$ be the universal quotient curve by $h$. Let $ft$, $ev$, $i$ denote the structure maps of $\mathcal{C}$ (the unitalicized such maps denote the ones coming from the ambient space $X_{g,n,d}$). Let the vertex and edge nodes of $\mathcal{C}$ be labelled $Z_v$, respectively, and label the cotangent branches by $L_\pm$. Any hatted version of the previously introduced notation refers to the corresponding construction on $\widehat{\mathcal{C}}$. \par

We have
$$\text{ft}_*\text{i}_*\text{ev}^*V_k|_{\mathcal{\widehat{M}}}=ft_*i_*\mathcal{\mathcal{O}_{Z_\mathcal{C}}}F_k(L_+,L_-)Eu(N)),$$ for $N$ some excess normal bundle bundle, and $Eu$ the $K$-theoretic Euler class.\par
$N_{Z_v}$ is trivial, since all vertex nodes can be smoothed within the stratum. This allows us to recast the nodal twisting restricted to 1-vertex strata solely in terms of the nodal loci of those strata.Let $\mathcal{\widehat{M}}$ now denote a stratum with one vertex and no edges. \par
If we denote the twisting class by $S$, the vertex potential is the cohomological potential of $X\times B\Z_M$, twisted by $\ch(tr_{\tilde{h}}(S|_\mathcal{\widehat{M}}))$, $\td(T\mathcal{\widehat{M}})$, and the denominator of the KRR formula, which contributes a class $\mathcal{O}^{tw}_{X\times B\Z_M}(T\mathcal{\widehat{M}})$. If $V$ denotes the terms coming from the inputs in a particular twisted correlator, the contribution of of $\mathcal{\widehat{M}}$ into the Kawasaki-Riemann-Roch formula applied to that correlator is: $$\chi^{fake}(\mathcal{\widehat{M}};\mathcal{O}^{vir}_{\mathcal{\widehat{M}}}\cdot tr_{\tilde{h}}(S\cdot V)\cdot \mathcal{O}^{tw}_{X\times B\Z_M}(T\mathcal{\widehat{M}})).$$ \par
We will henceforth isolate the contribution of the locus of nodes with $r$ copies on the covering curve, which we refer to as $Z_r$.\par
Differentiating the twisting class in $E_k$ brings down the factor $$\ch(\Delta_k^r)=\ch(tr_{\tilde{h}}\frac{\Psi^k}{k}(ft_*i_*ev^*F_k(L_+,L_-))).$$ \par
Since taking the (genuine) $K$-theoretic pushforward from the quotient $\widehat{Z}_r=Z_r//\Z_M$ extracts $\Z_M$-invariants, we can rewrite this expression as $\ch(\sum_{\lambda^M=1} \lambda^k\frac{\Psi^k}{k}(\widehat{ft}_*\widehat{i}_*\widehat{ev}^*F_k(L_+L_-)\otimes \C_{\lambda^{-1}})$ \par
To simplify the expression, we make the following calculation:
\begin{lem}
$\ch(\widehat{ft}_*\widehat{i}_*\widehat{ev}^*\alpha L_+^aL_-^b\otimes \C_{\lambda^{-1}}))=\begin{cases}0 &\lambda^r\neq \zeta^{a-b}\\ \widehat{ft}_*\widehat{i}_*(\ch(\widehat{ev}^*\alpha L_+^aL_-^b)\td(L_+L_-^*))&\lambda^r=\zeta^{b-a}\end{cases}$
\end{lem}
\begin{proof}
We apply Toen's Grothendieck-Riemann-Roch theorem. The preimage of $\mathcal{\widehat{M}}$ in the inertia stack of $Z_r$ is $m$ copies of the node, labelled by elements of the automorphism group of the node, labelled by powers of $\tilde{h}^r$, which acts on $L_\pm$ by $\zeta^{\mp 1}$. Since $L_+L_-$ is invariant under the $\tilde{h}^r$-action, the Todd class is invariant under $\tilde{h}^r$.\par 
So the pushforward is equal to $\widehat{ft}_*\widehat{i}_*\ch(\sum_s \lambda^{-rs}\zeta^{(b-a)s}\widehat{ev}^*\alpha L_+^aL_-^b)\td(L_+L_-)$.\par 
Since $\sum_{s=1}^m \lambda^{-rs}\zeta^{(b-a)s}=0$ unless $\lambda^{-r}\zeta^{b-a}=1$, so $\lambda^{r}=\zeta^{b-a}$, in which case the result is: $$\widehat{ft}_*\widehat{i}_*\ch(\widehat{ev}^*\alpha L_+^aL_-^b)\td(L_+L_-)$$\par
The factor $m$ from the $m$ copies is cancelled by the factor $\frac{1}{m}$ in the construction of the Chern character for orbifolds due to the size of the automorphism group at the node. 
\end{proof}
Since the Chern character intertwines Adams operations and cohomological power operations (denoted here $P^k$), the contribution of the term $\widehat{ev}^*\alpha L_+^aL_-^b$ of $\widehat{ev}^*F_k$ to $\ch(\Delta_k^r)$ can be described as the following cohomological pushforward:\par
\[\sum_{\lambda^{M}=1,\lambda^{r}=\zeta^{b-a}}\lambda^k\frac{P^k}{k}\big((\widehat{ft}\circ \widehat{i})_*\ch(\widehat{ev}^*aL_+^aL_-^b\otimes \C_{\lambda^{-1}})\td(L_+L_-)\big)\]\par

If $\lambda^{r}=\zeta^{b-a}$ and $\lambda$ is an $M$th root of unity, we necessarily have that $r|M$. So we can relabel $k$ as $rl_0$. Collecting the $r$ terms corresponding to the eigenvalues with $\lambda^{r}=\zeta^{b-a}$ terms yields that the above expression is equal to:
\[\zeta^{l_0(b-a)}\frac{P^{rl_0}}{l_0}(\widehat{ft}_*\widehat{i}_*\ch(\widehat{ev}^*\alpha L_+^aL_-^b)\td(L_+L_-)\]\par
Since orbifold Gromov-Witten theory uses the cotangent lines $\widehat{L}_\pm$ on the quotient curve, we rewrite $L_\pm$ as $\widehat{L}_\pm^{\frac{1}{m}}$, which is valid in fake $K$-theory even though such a bundle may not exist genuinely. Pulling back $\Delta_{rl_0}^r$ to $\widehat{Z}_r$, renaming $\frac{\Psi^{l_0}}{l_0}(F_k)$ to $S_k$, and reverting to the notation of fake $K$-theory yields:  \par
$$\widehat{i}^*\widehat{ft}^*\Delta_{rl_0}^{r}=\Psi^r\big(\widehat{ev}^*S_k(\zeta^{-1}\widehat{L}_+^{\frac{1}{m}},\zeta \widehat{L}_-^{\frac{1}{m}})(1-\widehat{L}_+^{\frac{l_0}{m}}\widehat{L}_-^{\frac{l_0}{m}})\big)$$\par
The factor $(1-\widehat{L}_+^{1/m}\widehat{L}_-^{1/m})$ occurs from pulling back $\widehat{i}_*\widehat{ft}_*\mathcal{O}_{\widehat{Z}_r}$, and is the $K$-theoretic Euler class of the normal bundle of $\widehat{Z}_r$ in $\mathcal{\widehat{M}}$.  \par
To compute the correlator as an integral on $\widehat{Z}_r$, we use the general formula for a morphism $Y\to X$: $$\chi^{fake}(X;V)=\chi^{fake}(Y;\frac{f^*(V)}{Eu(N_f)}).$$\par 
If we label the twisting class, contributions from the KRR denominators, and correlator inputs together as $B$, we thus have:
$$\chi^{fake}(\mathcal{\widehat{M}}; \Delta_k^r\cdot B\cdot \mathcal{O}^{tw}_{X\times B\Z_M})=$$\\
$$\chi^{fake}\big(\widehat{Z}_r;\frac{\Psi^r(S_k(\zeta^{-1}\widehat{L}_+^{1/m},\zeta \widehat{L}_-^{1/m})(1-\widehat{L}_+^{l_0/m}\widehat{L}_-^{l_0/m}))\cdot \widehat{i}^*\widehat{ft}^*(B\cdot \mathcal{O}^{tw}_{X\times B\Z_M})}{1-\widehat{L}_+^{1/m}\widehat{L}_-^{1/m}}\big).$$\par

Ungluing the nodes and integrating over the moduli spaces of component curves yields an order-2 recurrence relation on the correllators, in which the tensor $\frac{S_k(\zeta^{-1}L_+,\zeta L_-)(1-L_+^{l_0}L_-^{l_0})}{1-L_+L_-}$ is split among the points that were unglued and inserted in the corresponding seats. However, since we need the virtual structure sheaves of the components to match $\mathcal{O}^{tw}$ (i.e. also include the KRR denominators), we must also replace the denominator $1-L_+L_-$ with $\Psi^r(1-L_+L_-)$. This follows from the explicit calculation of $\mathcal{O}^{tw}$ in \cite{PermIX}, and accounts for the fact that deformations of a node on the quotient curve correspond to coherent deformations of the $r$ preimages on the covering curve, whereas in general they can be deformed independently. \par
A more detailed account of how ungluing the nodes interacts with cohomological nodal twisting classes is given in \cite{TwistedOrb} (see Proposition 3.9) for the case where $F_k$ are constants, the addition of nonconstant terms does not alter the argument. \par
The differential operator determined from this recurrence adds a factor of $\hbar^r/2$, due to the symmetry between $L_+$ and $L_-$, and the genus reduction (one node on the quotient curve corresponds to $r$ nodes on the covering curve).  \par
So the potential $\mathcal{D}^{tw,\mathbf{E}}_{X\times B\Z_M}$ satisfies the same differential equation as $\underline{\nabla}_r \mathcal{D}^{tw}_{X\times B\Z_M}$, where $\underline{\nabla}_r$ corresponds to changing the polarization in the sectors of order $r$ and eigenvalue $\zeta$ using the expression $\Psi^r\big(e^{\sum \frac{\Psi^{l}}{l}F_{rl}(\zeta^{-1}\widehat{L}_+^{\frac{1}{m}},\zeta\widehat{L}_-^{\frac{1}{m}})})$. 

\subsubsection{Edge contributions}
Recall that an edge in the graph of a Kawasaki stratum corresponds to an unbalanced node in the quotient curve corresponding to $r$ nodes on the cover curve where $\tilde{h}^r$ acts on the tangent branches with eigenvalues $\nu_+,\nu_-$, which are respectively primitive $m_+$,$m_-$, roots of unity, let $M$ be the order of $h$ on the stratum, and let $m=\frac{M}{r}$. \par
Fixing a particular edge $e_0$, we perform the same procedure as the vertices to compute the contribution of the nodal locus $Z_{e_0}$. The Euler factor $Eu(N)$ in the previous section becomes $1-L_+L_-$, since smoothing the edge node is normal to $\mathcal{\widehat{M}}$. \par
Differentiating in $E_k$ as before brings out the term $$\ch(\Delta_k^{e_0})=\ch(tr_h\frac{\Psi^k}{k}(ft_*i_*ev^*F_k(L_+L_-)(1-L_+L_-))=\ch(\sum_{\lambda^M=1}\lambda^k(\frac{\Psi^k}{k}(\widehat{ft}_*\widehat{i}_*F_k(L_+,L_-)(1-L_+L_-)).$$\par
The map $\widehat{ft}\circ\widehat{i}$ is an isomorphism on coarse spaces, since every point in $\mathcal{\widehat{M}}$ has a node corresponding to the edge. At the level of stacks, the automorphism group of the node is contracted to the identity, thus the (genuine) $K$-theoretic pushforward only extracts $h^r$ invariants. The term $\widehat{ev}^*L_+^iL_-^j\otimes \C_{\lambda^{-1}}$ only has a nonzero contribution when $\lambda^r=\mu^{-i}\nu^{-j}$. \par
Thus if $k=rl_0$, then $$\widehat{i}^*\widehat{ft}^*\Delta_k^e=\Psi^{r}(\widehat{ev}^*S^k(L_+\mu^{-1},L_-\nu^{-1})(1-\mu^{-1}\nu^{-1}L_+L_-).$$\par
This means that ungluing the edge nodes is done by applying the operator:
$e^{\sum_{edges}r\Psi^{r}(\hbar/2\underline{\nabla}_{\mu,\nu})}$, where $$\underline{\nabla}_{\mu,\nu}=\frac{e^{\sum_{l} \Psi^{l}{l}(F_{rl}(\widehat{L}_+^{\frac{1}{m_+}}\mu^{-1},
\widehat{L}_-^{\frac{1}{m_-}}\nu^{-1})}(\phi^\alpha\otimes\phi_\alpha)}{1-\mu^{-1}\nu^{-1}\widehat{L}_+^{\frac{1}{m_+}}
\widehat{L}_-^{\frac{1}{m_-}}}.$$
The other ingredients here are the same as the ones calculated in \cite{PermIX}: The denominator is the contribution of the normal bundle of $\mathcal{\widehat{M}}$ in the denominator of Kawasaki-Riemann-Roch formula, $\phi^\alpha,\phi_\alpha$ constitute a Poincare-dual basis of $K^0(X)$, which unglues the diagonal constraint at the nodes. \par
The resulting change of polarization on the adelic map pulls back to the one described in the theorem statement. 
\section*{Acknowledgements}
The author thanks Alexander Givental for suggesting this problem, and for his patience and guidance. This material is based upon work supported by the National Science Foundation Graduate Research Fellowship under Grant No. DGE1752814.

\printbibliography

\typeout{get arXiv to do 4 passes: Label(s) may have changed. Rerun}
\end{document}